\documentclass{article}
\usepackage[utf8]{inputenc}

\usepackage[utf8]{inputenc} 
\usepackage[T1]{fontenc}    
\usepackage{hyperref}       
\usepackage{url}            
\usepackage{booktabs}       
\usepackage{amsfonts}       
\usepackage{nicefrac}       
\usepackage{microtype}      
\usepackage{xcolor}         

\usepackage[utf8]{inputenc}
\usepackage{algorithm}
\usepackage{algorithmic}
\usepackage{natbib}
\usepackage{amsthm}
\usepackage{amssymb}
\usepackage{bbm}
\usepackage{amsmath}
\usepackage{graphicx}
\usepackage{subfigure}
\usepackage{xcolor}
\usepackage{hhline}

\colorlet{linkequation}{blue}

\newtheorem{thm}{Theorem}[section]

\newtheorem{defn}[thm]{Definition}
\newtheorem{rem}[thm]{Remark}

\newcommand{\A}{\mathcal{A}}

\newcommand{\E}{\mathbb{E}}

\newcommand{\R}{\mathbb{R}}
\newcommand{\K}{\mathcal{K}}
\newcommand{\F}{\mathcal{F}}
\renewcommand{\P}{\mathbb{P}}

\newcommand{\Scal}{\mathcal{S}}

\advance \oddsidemargin by -0.8in
\newcommand{\myparskip}{3pt}
\parskip \myparskip

\theoremstyle{plain}
\newtheorem{theorem}{Theorem}[section]
\newtheorem{lemma}[theorem]{Lemma}
\newtheorem{proposition}[theorem]{Proposition}

\newtheorem{claim}[theorem]{Claim}

\theoremstyle{definition}
\newtheorem{definition}[theorem]{Definition}

\title{The Convergence Rate of SGD's Final Iterate:\\ Analysis on Dimension Dependence}
\author{
  Daogao Liu\\
  University of Washington\\
  \texttt{dgliu@uw.edu} \\
  \and
  Zhou Lu\footnote{Alphabetical order. This work is done during both authors' visit at SQZ institution.}\\
  Princeton University \\
  \texttt{zhoul@princeton.edu} 
}
\date{June 2021}

\begin{document}

\maketitle

\begin{abstract}
Stochastic Gradient Descent (SGD) is among the simplest and most popular methods in optimization. The convergence rate for SGD has been extensively studied and tight analyses have been established for the running average scheme, but the sub-optimality of the final iterate is still not well-understood. 
\cite{shamir2013stochastic} gave the best known upper bound for the final iterate of SGD minimizing non-smooth convex functions, which is $O(\log T/\sqrt{T})$ for Lipschitz convex functions and $O(\log T/ T)$ with additional assumption on strongly convexity. The best known lower bounds, however, are worse than the upper bounds by a factor of $\log T$. \cite{harvey2019tight} gave matching lower bounds but their construction requires dimension $d= T$. It was then asked by \cite{koren2020open} how to characterize the final-iterate convergence of SGD in the constant dimension setting. 

In this paper, we answer this question in the more general setting for any $d\leq T$, proving $\Omega(\log d/\sqrt{T})$ and $\Omega(\log d/T)$ lower bounds for the sub-optimality of the final iterate of SGD in minimizing non-smooth Lipschitz convex and strongly convex functions respectively with standard step size schedules. 
Our results provide the first general dimension dependent lower bound on the convergence of SGD's final iterate, partially resolving a COLT open question raised by \cite{koren2020open}.
We also present further evidence to show the correct rate in one dimension should be $\Theta(1/\sqrt{T})$, such as a proof of a tight $O(1/\sqrt{T})$ upper bound for one-dimensional special cases in settings more general than \cite{koren2020open}.

\end{abstract}

\section{Introduction}
Stochastic gradient descent (SGD) is one of the oldest, simplest and most popular methods in optimization, dating back to \cite{robbins1951stochastic}. SGD works by iteratively takes a small step in the opposite direction of an unbiased estimate of sub-gradients, widely used in minimizing convex function $f$ over a convex domain $X$. Formally speaking, given a stochastic gradient oracle that for an input $x\in X$, SGD returns a random vector $\hat{g}$ whose expectation is equal to one of the sub-gradients of $f$ at $x$, and given an initial point $x_1$ generates a sequence of points $x_1,...,x_{T+1}$ according to the update rule
\begin{equation}
    x_{t+1}=\Pi_X(x_t-\eta_t \hat{g}_t)
\end{equation}
where $\Pi_X$ denotes projection onto $X$ and $\{\eta_t\}_{t\geq1}$ is a sequence of step sizes. 
Common choices of step sizes for convex functions are $\eta_t=1/\sqrt{t}$ for unknown horizon $T$ and $\eta_t=1/\sqrt{T}$ for known $T$, and $\eta_t=1/t$  for strongly convex functions. In both cases, it's known that the final-iterate convergence rate of SGD is optimal when $f$ is smooth \cite{nemirovski2009robust} or a running average scheme is employed. However, in practice the convex functions that arise are often non-smooth for example, in \cite{cohen2016geometric,lee2013new} and the final iterate is very often a more preferred choice than the running average as pointed out by \cite{shalev2011pegasos}. Nevertheless, the convergence rate of SGD's final iterate in the non-smooth setting is much less explored. Understanding this problem is important, because if the last iterate of SGD performs as good as the running average, it would yield a very simple, implementable and interpretable form of SGD. 

There is a line of works making attempts to understand the convergence rate of the final iterate of SGD. \cite{shamir2013stochastic} first established a near-optimal $O(\log T/\sqrt{T})$ convergence rate for the final iterate of SGD with a step size schedule $\eta_t=1/\sqrt{t}$. \cite{jain2019making} proved an information-theoretically optimal $O(1/\sqrt{T})$ upper bound using a rather non-standard step size schedule. \cite{harvey2019tight} gave an $\Omega(\log T/\sqrt{T})$ lower bound for the standard $\eta_t=1/\sqrt{t}$ step size schedule, but their construction requires the dimension $d$ to be equal to $T$, which is quite restrictive. A natural question arises:

\textbf{Question:} \emph{What's the dependence on dimension $d$ of the convergence rate of SGD's final iterate when $d\leq T$ is seen as a parameter?}

In a recent COLT open question raised by \cite{koren2020open}, the same question was posed but only for the more restrictive constant dimension setting. Moreover, they conjectured that the right convergence rate of SGD in the {\bf constant} dimensional case is $\Theta(1/\sqrt{T})$.
They analyzed a one-dimensional one-sided random walk special case as the preliminary evidence for their conjecture. However, this result is limited in the one-dimension setting for a special absolute-value function, thus can't be easily generalized. Analyzing the final-iterate convergence rate of SGD in general dimension for general convex functions is a more interesting and challenging question.
In particular, in \cite{koren2020open}, they wrote:

{\em For dimension $d>1$, a natural conjecture is that the right convergence rate is $\Theta(\log d/\sqrt{T})$, but we
have no indication to corroborate this.}

Motivated by this, we focus on analyzing the final iterate of SGD in general dimension $d\le T$ without smoothness assumptions in this paper.

Our first main result is an $\Omega(\log d/ T)$ lower bound for SGD minimizing strongly convex functions with a $\eta_t=1/t$ step size schedule when dimension $d\le T$, generalizing the result in \cite{harvey2019tight}. Our main observation is that we can let the initial point $x_1$ stay still for any number of steps as long as $\mathbf{0}$ is one of the sub-gradient of $f$ at $x_1$. By correctly modifying the original construction of \cite{harvey2019tight}, we can keep $x_1$ at $\mathbf{0}$ for $T-d$ steps and then 'kick' it to start taking a similar route as in \cite{harvey2019tight} in a $d$-dimensional space, which incurs an $\Omega(\log d/ T)$ sub-optimality since the logarithmic term is caused by the construction of $f$, not by taking the sum of step sizes. This result is then generalized to general Lipschitz convex functions, with either fixed $1/\sqrt{T}$ step size schedule or $1/\sqrt{t}$ decreasing step size schedule. Our lower bounds are actually proven for the sub-gradient descent method (GD), in a stronger form than SGD. The case of running SGD on strongly convex functions with fixed step sizes isn't considered because the fixed step size schedule is believed to be sub-optimal, therefore not commonly used. Roughly speaking, the sub-optimality with fixed step size $\eta$ is $O(\eta+(1-\eta)^T)$ according to \cite{bottou2018optimization}, and can't be tuned to the better rate $O(1/T)$ which can be attained by the $1/t$ step size schedule instead.

We also present further evidence on the upper bound for one dimensional special cases. Though seemingly easy, even the convergence rate of fixed-stepsize SGD for one-dimensional linear functions is open and non-trivial. \cite{koren2020open} considered minimizing a linear function with a restricted SGD oracle which only outputs $\pm 1$, reducing this problem to a one-sided random walk. We relax the restriction on the SGD oracle and prove an $O(1/\sqrt{T})$ optimal rate for (nearly) linear functions with the help of martingale theory. 
We also relax the linearity condition and prove the induced discrete random walk has $O(1/\sqrt{T})$ optimal rate in an asymptotic manner. These results serve as further evidence for the conjecture that the true rate of SGD in one dimension is $O(1/\sqrt{T})$.

Our contributions are summarized as follows:
\begin{itemize}
\setlength{\itemsep}{0pt}
    \item We prove an $\Omega(\log d/ T)$ lower bound for the sub-optimality of the final iterate of SGD minimizing non-smooth strongly convex functions with $\eta_t=1/ t$ step size schedule. We also prove an $\Omega(\log d/ \sqrt{T})$ lower bound for the sub-optimality of the final iterate of SGD minimizing non-smooth general Lipschitz convex functions with decreasing $\eta_t=1/\sqrt{t}$ step size schedule or $\eta_t=1/\sqrt{T}$ fixed step size schedule. Our results are the first, to the best of our knowledge, that characterize the general dimension dependence in analyzing the final iterate convergence of SGD.
    \item We also prove an optimal $O(1/\sqrt{T})$ upper bound for the sub-optimality of the final iterate of SGD minimizing Lipschitz convex functions with fixed $\Theta(1/\sqrt{T})$ step sizes in one dimension, under weaker assumptions than \cite{koren2020open}, that the function is (nearly) linear but allows the using of any legal oracle instead of a restricted one which only outputs $\pm 1$.
\end{itemize}

\begin{figure}[h]
\centering
\includegraphics[scale=.5]{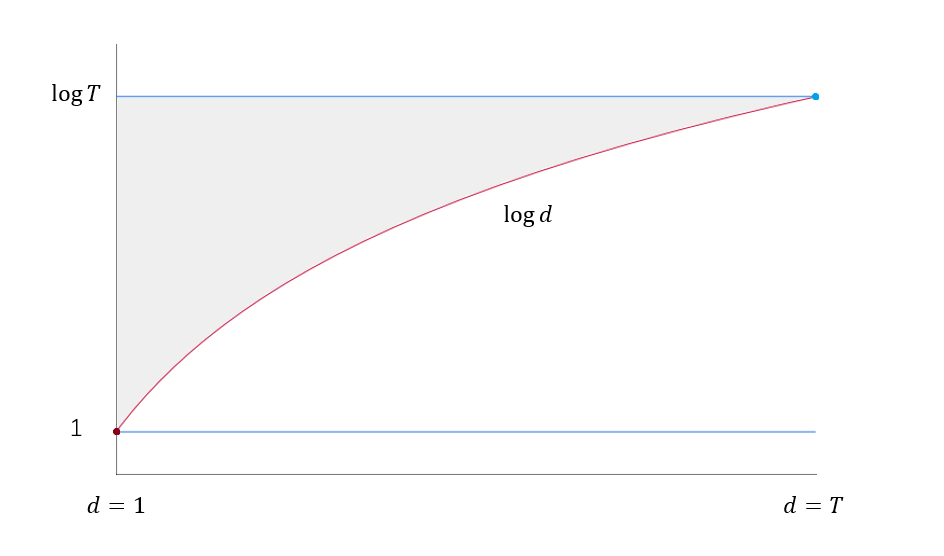}
\caption{The two blue lines denote previously known best upper/lower bounds for the sub-optimality of the final iterate of SGD, where we omit the $1/T, 1/\sqrt{T}$ terms. The blue dot on the upper right corner denotes the $O(\log T/ \sqrt{T})$ lower bound in \cite{harvey2019tight}, for the special case $d=T$. Our dimension-dependent lower bound is the red curve, with the red dot on the lower left corner denoting the tight upper bound for one-dimensional special cases. The grey region represents the unsolved upper bound. Our results partially solve the open question raised by \cite{koren2020open}, and provide further evidence that the true optimal rate is $\Theta(\log d/\sqrt{T})$.}
\end{figure}

\subsection{Related works}
Stochastic gradient descent (SGD) was first introduced by \cite{robbins1951stochastic}. It soon became one of the most popular tools in applied machine learning \cite{johnson2013accelerating, schmidt2017minimizing} due to its simplicity and effectiveness. Theoretical analysis on SGD usually adopts a running average step size schedule, which was first introduced by \cite{polyak1992acceleration} for optimal rates of convergence in the stochastic approximation setting. Optimal convergence rates have been achieved in both convex and strongly convex settings when averaging of iterates is used \cite{nemirovskij1983problem,zinkevich2003online,kakade2008generalization,cesa2004generalization}. The final iterate of SGD, though being a more preferred choice of step size schedule in practice, has not been very well studied from the theoretical perspective, and convergence results for the final iterate is rather scarce compared with the running average schedule.

\cite{shamir2013stochastic} first considered the question of the final iterate and gave a bound of $O(\log T/T)$ and $O(\log T/\sqrt{T})$ in expectation for the strongly convex case and Lipschitz case respectively, then high probability analogous upper bounds were provided in \cite{harvey2019tight}. However, there is still a $\log T$ gap from the optimal rate, and \cite{harvey2019tight} showed a matching lower bound implying the $\log T$ is inevitable. Nevertheless, their lower bound analysis relies on a construction with dimension $d=T$. \cite{jain2019making} used a sophisticated but non-standard step size schedule to achieve an optimal convergence rate for the final iterate of SGD.

Recently, \cite{koren2020open} asked if a dimension-dependent analysis can be made for the convergence of SGD's last iterate in the setting when $d$ is a constant. They conjectured that SGD in one dimension can achieve the optimal rate $O(1/\sqrt{T})$ using standard step size schedules. They also made a natural conjecture that the right convergence rate is $\Theta(\log d/\sqrt{T})$ for $d>1$, but with no indication to corroborate this. They considered an absolute value function in one dimension with fixed step size and a restricted oracle as the preliminary evidence for their first conjecture, by reducing SGD to a one sided random walk and using generating functions for analysis.

\begin{table}
\label{tab}
\centering
\begin{tabular}{|c|c|c|c|c|c|}
\hline 
Work & Rate & Method & Convexity &  Step size& Assumptions\\
\hline
\cite{nemirovski2009robust}& $O(1/T)$ & SGD & Strongly  & $1/t$& \\
\hline
\cite{jain2019making}& $O(1/\sqrt{T})$ & SGD & Convex  & Non-standard& \\
\hline
\cite{jain2019making}& $O(1/T)$ & SGD & Strongly & Non-standard& \\
\hline
\cite{shamir2013stochastic}& $O(\log T/\sqrt{T})$ & SGD & Convex  & $1/\sqrt{t}$& \\
\hline
\cite{shamir2013stochastic}& $O(\log T/T)$ & SGD & Strongly  & $1/t$& \\
\hline
\cite{harvey2019tight} & $\Omega(\log T/\sqrt{T})$ & GD & Convex  & $1/\sqrt{t}$& $d=T$ \\
\hline
\cite{harvey2019tight} & $\Omega(\log T/T)$ & GD & Strongly & $1/t$ & $d=T$\\
\hline
Ours & $\Omega(\log d/\sqrt{T})$ & GD & Convex  & $1/\sqrt{t}, 1/\sqrt{T}$& \\
\hline
Ours & $\Omega(\log d/T)$ & GD & Strongly & $1/t$ &\\
\hline
Ours & $O(1/T)$ & SGD & Special &  $1/\sqrt{T}$& $d=1$\\
\hline

\end{tabular}
\caption{Convergence results for the expected sub-optimality of the final iterate of SGD for minimizing non-smooth convex functions in various settings. GD denotes the sub-gradient descent method. Upper bounds of SGD automatically hold for GD, and lower bounds of GD hold for SGD as well since GD is a subset of SGD. Results for the running average scheme are not listed here for clarity of presentation.}
\end{table}

\subsection{Organization}
The settings and background knowledge are presented in Section 2. In Section 3 we prove the main $\Omega(\log d/T)$ lower bound for strongly convex functions with decreasing step sizes. We then extend this result to general Lipschitz convex functions with either decreasing or fixed step sizes. Section 4 focuses on one-dimensional special cases, proving a tight $O(1/\sqrt{T})$ upper bound for (nearly) linear functions (with any legal oracle). Section 5 concludes this paper.
\section{Preliminaries}
Let $X\subset \R^d$ be a closed and convex set, and a convex function $f: X\to \R$ defined on $X$, our goal is to solve $\min_{x\in X} f(x)$. In optimization, there is no explicit representation of $f$. Instead, we are allowed to use a stochastic oracle to query the sub-gradients of $f$ at $x$. The set $X$ is given in the form of a projection oracle, that outputs the closest point in $X$ to a given point $x$ in Euclidean norm. We introduce several standard definitions.

\begin{defn}[Sub-gradient]
A sub-gradient $g\in\R^d$ of a convex function $f: X\to \R$ at point $x$, is a vector satisfying that for any $x'\in X$, we have:
\begin{equation}
    f(x')-f(x)\ge g^{\top} (x'-x)
\end{equation}
we use $\partial f(x)$ to denote the set of all sub-gradients of $f$ at $x$.
\end{defn}

\begin{defn}[Strong Convexity]
A function $f: X\to \R$ is said to be $\alpha$-strongly convex, if for any $x,y\in X$ and $g\in \partial f(x)$, the following holds:
\begin{equation}
    f(y)-f(x)\ge g^{\top}(y-x) +\frac{\alpha}{2}||y-x||^2
\end{equation}
\end{defn}

\begin{defn}[Lipschitz Function]
A function $f: X\to \R$ is called $L$-Lipschitz (with respect to $\ell_2$ norm), if for any $x,y\in X$, we have that:
\begin{equation}
    |f(x)-f(y)|\le L||x-y||_2
\end{equation}
Further, if we assume $f$ is convex, the above definition is equal to $||g||_2\le L$ for any sub-gradient $g$.
\end{defn}

Let $\Pi_X$ denote the projection operator on $X$, the (projected) stochastic gradient descent (SGD) is given in Algorithm \ref{alg1}, in the most standard form except for the output. The choice of running average output enjoys optimal convergence rates \cite{polyak1992acceleration,rakhlin2011making,ruppert1988efficient}. However, the more popular choice in practice is simply using the final iterate as output.

There are also several choices for step size schedule $\eta_t$. The optimal choice of step size is known to be $\eta_t=1/t$ for strongly convex functions and $\eta_t=1/\sqrt{t}$ for Lipschitz convex functions when the horizon isn't (necessarily) known in advance. When $T$ is known, we can also choose $\eta_t=1/\sqrt{T}$ for Lipschitz convex functions.

\begin{algorithm}[h]
\caption{Stochastic gradient descent with the final iterate output} 
\begin{algorithmic}[1] \label{alg1}
\STATE Given $X \subset \R^d$, initial point $x_1\in X$, step size schedule $\eta_t$:
\FOR{$j=1,...,T$:}
\STATE Query stochastic gradient oracle at $x_t$ for $\hat{g}_t$ such that $\E[\hat{g}_t|\hat{g}_1,...,\hat{g}_{t-1}]\in \partial f(x_t)$
\STATE $y_{t+1}=x_t-\eta_t \hat{g}_t$
\STATE $x_{t+1}=\Pi_{X}(y_{t+1})$
\ENDFOR
\RETURN $x_{T+1}$
\end{algorithmic} 
\end{algorithm}

\section{Main lower bounds}
In this section we prove our main result, that the final iterate of SGD for non-smooth strongly convex functions has sub-optimality $\Omega(\log d/T)$, even in the non-stochastic case. We modify the construction used in \cite{harvey2019tight} which proves an $\Omega(\log T/T)$ lower bound for the special case $d=T$. In a nutshell, we consider the setting $d\le T$ and construct a function $f$ along with a special sub-gradient oracle such that the initial point will stay still for the first $T-d$ steps then start moving in Algorithm~\ref{alg1}, in which the final iterate satisfies $f(x_{T+1})=\Omega(\log d/T)$. Then we extend the analysis to Lipschitz convex functions.

\subsection{Strongly convex functions}

Let $[j]$ be the set of positive integers that are no more than $j$.
For simplicity, we consider function which is 3-Lipschitz and 1-strongly convex over the Euclidean unit ball.
For general Lipschitz and strongly convexity, it is easy to scale our construction and get corresponding lower bounds.

\begin{theorem}[Main Result]\label{thmain}
For any $T$ and $1\leq d\leq T$, there exists a convex function $f: X\to \R$ where $X\subset \R^d$ is the Euclidean unit ball, and $f$ is 3-Lipschitz and 1-strongly convex. When executing Algorithm \ref{alg1} on $f$ with initial point 0 (the global minimum) and step size schedule $\eta_t=1/t$, the final iterate satisfies:
\begin{equation}
    f(x_{T+1})-\min_{x\in X} f(x)\ge \frac{\log d}{5T}
\end{equation}
\end{theorem}

\begin{proof}
Define $f:X \to \R$ and $h_i\in \R^d$ for $i\in[d+1] \cup \{0\}$ by
\begin{align*}
    f(x)=\max_{0\leq i\leq d+1}H_i(x)
\end{align*}
where $H_i(x)=h_i^\top x+\frac{1}{2}\|x\|^2$. For $i\ge 1$ we define
\begin{align*}
    h_{i, j}=\left\{\begin{array}{ll}
a_{j} & (\text { if } 1 \leq j<i) \\
-1 & (\text { if } i=j \leq d) \\
0 & (\text { if } i<j \leq d)
\end{array} \quad \text { and } \quad a_{j}=\frac{1}{2(d+1-j)} \quad(\text { for } j \in[d] )\right.
\end{align*}

Additionally, let $h_0=\mathbf{0}$ and $H_0(x)=\frac{1}{2}\|x\|^2$. It's easy to check that $f$ is 3-lipschitz and 1-strongly convex, with minimal value 0. We have the following standard claim \cite{hiriart2013convex}.

\begin{claim}
$\partial f(x)$ is the convex hull of $\{h_i+x\mid i\in {\cal I}(x)\}$ where ${\cal I}(x)=\{i\geq 0\mid H_i(x)=f(x)\}$.
\end{claim}

Our non-stochastic sub-gradient oracle outputs $\mathbf{0}$ for the first $T-d$ steps and outputs $h_{i'}+x$ where $i'=\min{\cal I}(x) \setminus \{0\}$ for the last $d$ steps. Define $z_1=\cdots =z_{T-d+1}=0$, let $T^*:=T-d$ and 
\begin{align*}
    z_{t,j}=\left\{\begin{array}{ll}
       \frac{1-(t-T^*-j-1)a_j}{t-1}  & (\text{ if } 1\leq j<t-T^*)  \\ 
        0 & (\text{ if } t-T^* \leq j \leq T)
    \end{array} \right. \quad (\text{ for } t>T^*+1).
\end{align*}

We will show inductively that these are precisely the first $T$ iterates produced by algorithm \ref{alg1} when using the sub-gradient oracle defined above. The following claim is easy to verify from definition.

\begin{claim}
We have the following claims:
\begin{itemize}
    \item $z_t$ is non-negative. In particular, $z_{t,j}\geq \frac{1}{2(t-1)}$ for $j<t-T^*$ and $z_{t,j}=0$ for $j\geq t-T^*$.
    \item $z_t=\mathbf{0}$ for $t\in [T^*+1]$ and $\|z_t\|^2\leq \frac{1}{t-1}$ for $t>T^*+1$. Thus $z_t\in X$ for all $t$.
\end{itemize}
\end{claim}

\begin{proof}
The first claim simply follows from the fact that $\frac{t-T^*-j-1}{d-j+1}\le 1$. The second claim follows from that $(t-T^*-1)\frac{1}{(t-1)^2} \le \frac{1}{t-1}$.
\end{proof}

We can now determine the value and sub-differential at $z_t$. The case for the first $T^*$ steps is trivial as the sub-gradient oracle always outputs $\mathbf{0}$ and $x_1$ never moves. For the last $d$ steps we have the following claim.
\begin{claim}
$f(z_t)=H_{t-T^*}(z_t)$ for all $T^*<t\leq T+1$. The subgradient oracle for $f$ at $z_t$ returns the vector $h_{t-T^*}+z_t$.
\end{claim}

\begin{proof}
We claim that $h_{t-T^*}^\top z_t=h_{i-T^*}^\top z_t$ for all $i>t>T^*$. 
By definition, $z_t$ is supported on its first $t-T^*$ coordinates, completing the first part of the claim. Next we claim that $z_t^\top h_{t-T^*}> z_t^\top h_{i-T^*}$ for all $T^*+1\leq i <t$. For $T^*\leq i<t$, one has
\begin{align*}
    z_t^\top (h_{t-T^*}-h_{i-T^*})=\sum_{j=i}^{t-1}z_{t,j}(h_{t-T^*,j}-h_{i-T^*,j})=z_{t,i}(a_i+1)+\sum_{j=i+1}^{t-1}z_{t,j}a_j>0.
\end{align*}

The two claims guarantee that $f(z_t)=H_{t-T^*}(z_t)$. Combining with the fact ${\cal I}(z_t)=\{t-T^*,...,d+1\}$, we conclude that the sub-gradient oracle outputs $h_{t-T^*}+z_t$.
\end{proof}

\begin{lemma}
For the function constructed in this section, the solution of $t$-th step in Algorithm \ref{alg1} equals to $z_t$ for every $T^*<t\leq T+1$.
\end{lemma}

\begin{proof}
We prove this lemma by induction.
For base case $t=T^*+1$, we know that $z_t=-\eta_t h_1$ holds.
Thus,
\begin{align*}
    y_{t+1,j}=& z_{t,j}-\frac{1}{t}(h_{t-T^*,j}+z_{t,j})\\
             =& \frac{t-1}{t} \left\{ \begin{array}{ll}
                 \frac{1-(t-T^*-j-1)a_j}{t-1} & (\text{ for }1\leq j<t-T^*)  \\
                 0 & (\text{ for } j\geq t-T^*)
             \end{array} \right\}
             -\frac{1}{t}\left\{
\begin{array}{ll}
a_{j} & (\text { if } 1 \leq j<t-T^*) \\
-1 & (\text { if } t-T^*=j \leq d) \\
0 & (\text { if } t-T^*<j \leq d)
\end{array}\right\}\\
    =& \frac{1}{t} \left\{ \begin{array}{ll}
                 1-(t-T^*-j-1)a_j & (\text{ for }1\leq j<t-T^*)  \\
                 0 & (\text{ for } j\geq t-T^*)
             \end{array} \right\}
             -\frac{1}{t}\left\{
\begin{array}{ll}
a_{j} & (\text { if } 1 \leq j<t-T^*) \\
-1 & (\text { if } t-T^*=j \leq d) \\
0 & (\text { if } t-T^*<j \leq d)
\end{array}\right\}\\
    =& \left\{\begin{array}{ll}
        \frac{1-(t-T^*-j)a_j}{t} & (\text{ for } j<t-T^*)  \\
        \frac{1}{t} & (\text{ for }j=t-T^*)\\
        0 & (\text{ fro } j>t-T^*)
    \end{array}\right\}.
\end{align*}
So $y_{t+1}=z_{t+1}$. Since $z_{t+1}\in X$, we have that $x_{t+1}=z_{t+1}$.

\end{proof}

From the above claim we have that the vector $x_t$ in algorithm \ref{alg1} is equal to $z_t$ for $t\in [T+1]$, which allows determination of the value of the final iterate:
\begin{equation}\label{eq2}
    f(x_{T+1})=f(z_{T+1})=H_{d+1}(z_{T+1})\geq \sum_{j=1}^{d}h_{d+1,j}z_{T+1,j}\geq \sum_{j=1}^{d} \frac{1}{2(d+1-j)}\frac{1}{2T}>\frac{\log d}{5T}.
\end{equation}
\end{proof}

\begin{rem}
For the case $d=1$ we still have the $\Omega(1/T)$ lower bound, by not using $\sum_{i=1}^d \frac{1}{i}> \log d$ in the last step of equation \ref{eq2}.
\end{rem}

Theorem \ref{thmain} improves the previously known lower bound by a factor of $\log d$, implying an inevitable dependence on dimension of the convergence of SGD's final iterate. In other words, the final iterate of SGD performs worse than the running average by a factor of $\log d$ in convergence rate, combining with the upper bounds in \cite{nemirovski2009robust}. Our result is the first general dimension-dependent analysis, to the best of our knowledge, for the convergence rate of SGD's final iterate. Next we extend this result to Lipschitz convex functions.

\subsection{Lipschitz convex functions with $1/\sqrt{t}$ step sizes}
In this subsection we prove that the final iterate of SGD for non-smooth lipschitz convex functions has sub-optimality $\Omega(\log d/\sqrt{T})$. The setting and approach is similar to the strongly convex case except that we use a $\eta_t=1/\sqrt{t}$ step size schedule instead. Without loss of generality we consider only 1-lipschitz convex functions. 


\begin{theorem}
For any $T$ and $1\le d\leq T$, there exists a convex function $f: X\to \R$ where $X\subset \R^d$ is the Euclidean unit ball, and $f$ is 1-lipschitz. When executing algorithm \ref{alg1} on $f$ with initial point 0 and step size schedule $\eta_t=1/\sqrt{t}$, the last iterate satisfies:
\begin{equation}
    f(x_{T+1})-\min_{x\in X} f(x)\ge \frac{\log d}{32\sqrt{T}}
\end{equation}
\end{theorem}

\begin{proof}
Define $f:X \to \R$ and $h_i\in \R^d$ for $i\in[d+1] \cup \{0\}$ by
\begin{align*}
    f(x)=\max_{0\leq i\leq d+1}H_i(x)
\end{align*}
where $H_i(x)=h_i^\top x$. For $i\ge 1$ we define
\begin{align*}
    h_{i, j}=\left\{\begin{array}{ll}
a_{j} & (\text { if } 1 \leq j<i) \\
-b_i & (\text { if } i=j \leq d) \\
0 & (\text { if } i<j \leq d)
\end{array} \quad \text { and } \quad a_{j}=\frac{1}{8(d+1-j)},\quad b_j=\frac{\sqrt{j+T-d}}{2\sqrt{T}} \quad(\text { for } j \in[d] )\right.
\end{align*}

Additionally, let $h_0=\mathbf{0}$ and $H_0(x)=0$. It's easy to check that $f$ is 1-Lipschitz, with minimal value 0. We have the following standard claim.

\begin{claim}
$\partial f(x)$ is the convex hull of $\{h_i\mid i\in {\cal I}(x)\}$ where ${\cal I}(x)=\{i\geq 0\mid H_i(x)=f(x)\}$.
\end{claim}

Our non-stochastic sub-gradient oracle outputs $\mathbf{0}$ for the first $T-d$ steps and outputs $h_{i'}$ where $i'=\min{\cal I}(x) \setminus \{0\}$ for the last $d$ steps. Define $z_1=\cdots =z_{T-d+1}=\mathbf{0}$, let $T^*=:T-d$.

\begin{align*}
    z_{t,j}=\left\{\begin{array}{ll}
       \frac{b_j}{\sqrt{j+T^*}}-a_j\sum_{k=j+T^*+1}^{t-1}\frac{1}{\sqrt{k}}   & (\text{ if } 1\leq j< t-T^*) \\ 
        0 & (\text{ if } t-T^* \leq j \leq d)
    \end{array} \right. \quad (\text{ for } t>T^*+1).
\end{align*}

We will show inductively that these are precisely the first $T$ iterates produced by algorithm \ref{alg1} when using the sub-gradient oracle defined above. The following claim is obvious from definition.

\begin{claim}
We have the following claims:
\begin{itemize}
    \item $z_t$ is non-negative. In particular, $z_{t,j}\geq \frac{1}{4\sqrt{T}}$ for $j<t-T^*$ and $z_{t,j}=0$ for $j\geq t-T^*$.
    \item $z_{t,j}\leq \frac{1}{2\sqrt{T}}$ ofr all $j$. In particular, $z_t\in X$.
\end{itemize}
\end{claim}
\begin{proof}
It is obvious that $z_{t,j}=0$ for $j\geq t-T^*$ from the definition.
As $\frac{b_j}{\sqrt{j+T^*}}=\frac{1}{2\sqrt{T}}$, it suffices to prove that $0\leq a_j\sum_{k=j+T^*}^{t-1}\frac{1}{\sqrt{k}}\leq \frac{1}{4\sqrt{T}}$. We have that 
\begin{equation}
    0\leq \sum_{k=j+T^*}^{t-1}\frac{1}{\sqrt{k}}\le \int_{j+T^*-1}^{t-1} \frac{1}{\sqrt{x}} \mathrm{d} x=\frac{2(t-j-T^*)}{\sqrt{t-1}+\sqrt{j+T^*-1}}\le \frac{2(t-j-T^*)}{\sqrt{t-1}}
\end{equation}
and further $\frac{t-j-T^*}{\sqrt{t-1}}\le \frac{T+1-j-T^*}{\sqrt{T}}=\frac{d+1-j}{\sqrt{T}}$ by monotony.
Thus $0\leq a_j\sum_{k=j+T^*}^{t-1}\frac{1}{\sqrt{k}}\leq \frac{1}{4\sqrt{T}}$ follows from the definition of $a_j$.

\end{proof}

We can now determine the value and sub-differential at $z_t$. The case for the first $T^*$ steps is trivial as the sub-gradient oracle always outputs 0 and $x_1$ never moves a bit. For the last $d$ steps we have the following claim.
\begin{claim}
$f(z_t)=H_{t-T^*}(z_t)$ for all $T^*<t\leq T+1$. The sub-gradient oracle for $f$ at $z_t$ returns the vector $h_t$.
\end{claim}

\begin{proof}
We know that $h_{t-T^*}^\top z_t=h_{i-T^*}^\top z_t$ for all $i>t>T^*$. 
By definition, $z_t$ is supported on its first $t-T^*$ coordinates, completing the first part of the claim. Next we claim that $z_t^\top h_{t-T^*}> z_t^\top h_{i-T^*}$ for all $T^*+1\leq i <t$. For $T^*+1\leq i<t$, one has
\begin{align*}
    z_t^\top (h_{t-T^*}-h_{i-T^*})=\sum_{j=i-T^*}^{t-T^*}z_{t,j}(h_{t-T^*,j}-h_{i-T^*,j})=z_{t,i}(a_i+1)+\sum_{j=i+1}^{t-1}z_{t,j}a_j>0.
\end{align*}

The two claims guarantee that $f(z_t)=H_{t-T^*}(z_t)$. Combining with the fact ${\cal I}(z_t)=\{t-T^*,...,d+1\}$, we conclude that the sub-gradient oracle outputs $h_{t-T^*}$.
\end{proof}

\begin{lemma}
For the function constructed in this section, the solution of $t$-th step in algorithm \ref{alg1} equals to $z_t$ for every $T^*<t\leq T+1$.
\end{lemma}

\begin{proof}
We prove this lemma by induction.
For base case $t=T^*+1$, we know that $z_t=-\eta_t h_1$ holds.
Thus,
\begin{align*}
    y_{t+1,j}=& z_{t,j}-\frac{1}{\sqrt{t}}h_{t-T^*,j}\\
             =& \left\{ \begin{array}{ll}
                 \frac{b_j}{\sqrt{j+T^*}}-a_j\sum_{k=j+T^*}^{t-1}\frac{1}{\sqrt{k}} & (\text{ for }1\leq j<t-T^*)  \\
                 0 & (\text{ for } j\geq t-T^*)
             \end{array} \right\}
             -\frac{1}{\sqrt{t}}\left\{
\begin{array}{ll}
a_{j} & (\text { if } 1 \leq j<t-T^*) \\
-b_i & (\text { if } t-T^*=j \leq d) \\
0 & (\text { if } t-T^*<j \leq d)
\end{array}\right\}\\
    =& \left\{\begin{array}{ll}
        \frac{b_j}{\sqrt{j+T^*}}-a_j\sum_{k=j+T^*}^{t}\frac{1}{\sqrt{k}} & (\text{ for } j<t-T^*)  \\
        \frac{b_t}{\sqrt{t}}=\frac{b_t}{\sqrt{j+T^*}} & (\text{ for }j=t-T^*)\\
        0 & (\text{ fro } j>t-T^*)
    \end{array}\right\}.
\end{align*}
So $y_{t+1}=z_{t+1}$. Since $z_{t+1}\in X$, we have that $x_{t+1}=z_{t+1}$.
\end{proof}

From the above claim we have that the vector $x_t$ in algorithm \ref{alg1} is equal to $z_t$ for $t\in [T+1]$, which allows determination of the value of the final iterate:
\begin{align*}
    f(x_{T+1})=f(z_{T+1})=H_{d+1}(z_{T+1})\geq \sum_{j=1}^{d}h_{d+1,j}z_{T+1,j}\geq \sum_{j=1}^{d} \frac{1}{8(d+1-j)}\frac{1}{4\sqrt{T}}>\frac{\log d}{32\sqrt{T}}.
\end{align*}
\end{proof}
\subsection{Lipschitz convex functions with $1/\sqrt{T}$ step sizes}
In this section we prove that the final iterate of SGD for non-smooth Lipschitz convex functions has sub-optimality $\Omega(\log d/\sqrt{T})$ when a fixed $\eta_t=1/\sqrt{T}$ step size schedule is adopted.

\begin{theorem}
For any $T$ and $1\le d\leq T$, there exists a convex function $f: X\to \R$ where $X\subset \R^d$ is the Euclidean unit ball, and $f$ is 1-Lipschitz. When executing algorithm \ref{alg1} on $f$ with initial point 0 and step size schedule $\eta_t=1/\sqrt{T}$, the last iterate satisfies:
\begin{equation}
    f(x_{T+1})-\min_{x\in X} f(x)\ge \frac{\log d}{32\sqrt{T}}
\end{equation}
\end{theorem}

\begin{proof}
Define $f:X \to \R$ and $h_i\in \R^d$ for $i\in[d+1] \cup \{0\}$ by
\begin{align*}
    f(x)=\max_{0\leq i\leq d+1}H_i(x)
\end{align*}
where $H_i(x)=h_i^\top x$. For $i\ge 1$ we define
\begin{align*}
    h_{i, j}=\left\{\begin{array}{ll}
a_{j} & (\text { if } 1 \leq j<i) \\
-b_i & (\text { if } i=j \leq d) \\
0 & (\text { if } i<j \leq d)
\end{array} \quad \text { and } \quad a_{j}=\frac{1}{8(d+1-j)},\quad b_j=\frac{1}{2} \quad(\text { for } j \in[d] )\right.
\end{align*}

Additionally, let $h_0=0$ and $H_0(x)=0$. It's easy to check that $f$ is 1-Lipschitz, with minimal value 0. We have the following standard claim.

\begin{claim}
$\partial f(x)$ is the convex hull of $\{h_i\mid i\in {\cal I}(x)\}$ where ${\cal I}(x)=\{i\geq 0\mid H_i(x)=f(x)\}$.
\end{claim}

Our non-stochastic sub-gradient oracle outputs 0 for the first $T-d$ steps and outputs $h_{i'}$ where $i'=\min{\cal I}(x) \setminus \{0\}$ for the last $d$ steps. Define $z_1=\cdots =z_{T-d+1}=0$, let $T^*=:T-d$.
\begin{align*}
    z_{t,j}=\left\{\begin{array}{ll}
       \frac{b_j}{\sqrt{T}}-a_j\frac{t-j-T^*-1}{\sqrt{T}}   & (\text{ if } 1\leq j< t-T^*) \\ 
        0 & (\text{ if } t-T^* \leq j \leq d)
    \end{array} \right. \quad (\text{ for } t>T^*+1).
\end{align*}

We will show inductively that these are precisely the first $T$ iterates produced by algorithm \ref{alg1} when using the sub-gradient oracle defined above. The following claim is obvious from definition.

\begin{claim}
We have the following claims:
\begin{itemize}
    \item $z_t$ is non-negative. In particular, $z_{t,j}\geq \frac{1}{4\sqrt{T}}$ for $j<t-T^*$ and $z_{t,j}=0$ for $j\geq t-T^*$.
    \item $z_{t,j}\leq \frac{1}{2\sqrt{T}}$ ofr all $j$. In particular, $z_t\in X$.
\end{itemize}
\end{claim}
\begin{proof}
It is obvious that $z_{t,j}=0$ for $j\geq t-T^*$ from the definition.
As $\frac{b_j}{\sqrt{T}}=\frac{1}{2\sqrt{T}}$, it suffices to prove that $0\leq a_j\frac{t-j-T^*-1}{\sqrt{T}}\leq \frac{1}{4\sqrt{T}}$, which is direct as $0\leq t-j-T^*-1\le d+1-j$.

\end{proof}

We can now determine the value and sub-differential at $z_t$. The case for the first $T^*$ steps is trivial as the sub-gradient oracle always outputs 0 and $x_1$ never moves a bit. For the last $d$ steps we have the following claim.
\begin{claim}
$f(z_t)=H_{t-T^*}(z_t)$ for all $T^*<t\leq T+1$. The sub-gradient oracle for $f$ at $z_t$ returns the vector $h_t$.
\end{claim}

\begin{proof}
We know that $h_{t-T^*}^\top z_t=h_{i-T^*}^\top z_t$ for all $i>t>T^*$. 
By definition, $z_t$ is supported on its first $t-T^*$ coordinates, completing the first part of the claim. Next we claim that $z_t^\top h_{t-T^*}> z_t^\top h_{i-T^*}$ for all $T^*+1\leq i <t$. For $T^*+1\leq i<t$, one has
\begin{align*}
    z_t^\top (h_{t-T^*}-h_{i-T^*})=\sum_{j=i-T^*}^{t-T^*}z_{t,j}(h_{t-T^*,j}-h_{i-T^*,j})=z_{t,i}(a_i+1)+\sum_{j=i+1}^{t-1}z_{t,j}a_j>0.
\end{align*}

The two claims guarantee that $f(z_t)=H_{t-T^*}(z_t)$. Combining with the fact ${\cal I}(z_t)=\{t-T^*,...,d+1\}$, we conclude that the sub-gradient oracle outputs $h_{t-T^*}$.
\end{proof}

\begin{lemma}
For the function constructed in this section, the solution of $t$-th step in algorithm \ref{alg1} equals to $z_t$ for every $T^*<t\leq T+1$.
\end{lemma}

\begin{proof}
We prove this lemma by induction.
For base case $t=T^*+1$, we know that $z_t=-\eta_t h_1$ holds.
Thus,
\begin{align*}
    y_{t+1,j}=& z_{t,j}-\frac{1}{\sqrt{T}}h_{t-T^*,j}\\
             =& \left\{ \begin{array}{ll}
                 \frac{b_j}{\sqrt{T}}-a_j\frac{t-j-T^*-1}{\sqrt{T}} & (\text{ for }1\leq j<t-T^*)  \\
                 0 & (\text{ for } j\geq t-T^*)
             \end{array} \right\}
             -\frac{1}{\sqrt{T}}\left\{
\begin{array}{ll}
a_{j} & (\text { if } 1 \leq j<t-T^*) \\
-b_i & (\text { if } t-T^*=j \leq d) \\
0 & (\text { if } t-T^*<j \leq d)
\end{array}\right\}\\
    =& \left\{\begin{array}{ll}
        \frac{b_j}{\sqrt{T}}-a_j\frac{t-j-T^*-1}{\sqrt{T}} & (\text{ for } j<t-T^*)  \\
        \frac{b_t}{\sqrt{T}}=\frac{1}{2\sqrt{T}} & (\text{ for }j=t-T^*)\\
        0 & (\text{ for } j>t-T^*)
    \end{array}\right\}.
\end{align*}
So $y_{t+1}=z_{t+1}$. Since $z_{t+1}\in X$, we have that $x_{t+1}=z_{t+1}$.
\end{proof}

From the above claim we have that the vector $x_t$ in algorithm \ref{alg1} is equal to $z_t$ for $t\in [T+1]$, which allows determination of the value of the final iterate:
\begin{align*}
    f(x_{T+1})=f(z_{T+1})=H_{d+1}(z_{T+1})\geq \sum_{j=1}^{d}h_{d+1,j}z_{T+1,j}\geq \sum_{j=1}^{d} \frac{1}{8(d+1-j)}\frac{1}{4\sqrt{T}}>\frac{\log d}{32\sqrt{T}}.
\end{align*}
\end{proof}
\section{Evidence in one dimension}

\subsection{Assumptions}
Standard, we make the following assumptions ${\cal A}$ for running SGD:
\begin{itemize}
    \item The domain $X\subset \R$ is convex and bounded with diameter $D$\;
    \item The objective $f:X \rightarrow \R$ is convex and $G$-Lipschitz, not necessarily differentiable\;
    \item The output stochastic gradients are bounded: $|\hat{g}_t|\leq G$, and we have $\E[\hat{g}_t\mid \hat{g}_1,\cdots,\hat{g}_{t-1}]\in\partial f(x_t)$.
\end{itemize}

\begin{definition}
We define a set of good points by $\Scal$, which contains all ideal points:
\begin{align*}
    \Scal=\{x\in X:f(x)-f^*\leq \frac{GD}{\sqrt{T}}\}.
\end{align*}
\end{definition}

Moreover, we consider a special convex function family which we call Nearly Linear Function:
\begin{definition}[Nearly Linear Function]
We call a convex function $f:X\rightarrow \R$ nearly linear if it satisfies the following assumption ${\cal B}$:
\begin{itemize}
    \item For any $x_t\notin \Scal$, there exists constants $0<\epsilon\leq 1,0<c\leq 1$ such that
$\big|\E[\hat{g}_t\mid \hat{g}_1,\cdots,\hat{g}_{t-1}]\big|\in[c\epsilon G,\epsilon G]$.
\end{itemize}
\end{definition}

The family of nearly linear functions captures those functions whose sub-gradients do not change drastically outside the set of good points. The linear functions considered in \cite{koren2020open} lie in this family.

\subsection{Preliminary on Martingale}
We demonstrate some basic definitions and theorem with relationship to Martingale, which is used in the later proof.
\begin{definition}[Martingale]
A sequence $Y_1,Y_2,\cdots$ is said to be a martingale with respect to another sequence $X_1,X_2,\cdots$ if for all $n$:
\begin{itemize}
    \item $\mathbf{E}\left(\left|Y_{n}\right|\right)<\infty$
    \item $\mathbf{E}\left(Y_{n+1} \mid X_{1}, \ldots, X_{n}\right)=Y_{n}$.
\end{itemize}
\end{definition}

\begin{definition}[Martingale Difference]
Consider an adapted sequence $\{X_t,\F_t\}_{-\infty}^{\infty}$ on a probability space.
$X_t$ is a martingale difference sequence (MDS) if it satisfies the following two conditions for all $t$:
\begin{itemize}
    \item $\E|X_t|<\infty$
    \item $\E[X_t\mid \F_{t-1}]=0$, a.s.
\end{itemize}
\end{definition}

\begin{definition}[Stopping Time]
A stopping time with respect to a sequence of random variables $X_1,X_2,X_3,\cdots$ is a random variable $\tau$ with the property that for each $t$, the occurrence or non-occurrence of the event $\tau=t$ depends only on the values of $X_1,X_2,X_3,\cdots,X_t$.
\end{definition}

\begin{theorem}[Freedman's Inequality, Theorem 1.6 in~\cite{Fre75}]
\label{thm:FreedmanInequality}
Consider a real-valued martingale difference sequence $\{X_t\}_{t\geq 0}$ such that $X_0=0$, and $\E[X_{t+1}|\mathcal{F}_t]=0$ for all $t$, where $\{\mathcal{F}_t\}_{t\geq 0}$ is the filtration defined by the sequence. Assume that the sequence is uniformly bounded, i.e., $|X_t|\leq M$ almost surely for all $t$. Now define the predictable quadratic variation process of the martingale to be $W_t=\sum_{j=1}^t \E[X_j^2|\mathcal{F}_{j-1}]$ for all $t\geq 1$. Then for all $\ell \geq 0$ and $\sigma^2>0$ and any stopping time $\tau$, we have
\[
\Pr\Big[ \Big|\sum_{j=0}^\tau X_j \Big|\geq \ell \wedge W_\tau \leq \sigma^2 \text{for some stopping time } \tau \Big] \leq 2\exp\Big(- \frac{\ell^2/2}{\sigma^2+M \ell/3} \Big).
\]
\end{theorem}

\subsection{Analysis}
In this subsection, we show how to improve the convergence of the last iterate of SGD with a fixed step size $\eta=\frac{4D}{G\sqrt{T}}$ in one dimension for nearly linear functions.

The proof mainly consists of two parts.
In the first part, we prove that for running SGD with fixed step size for any convex function satisfying Condition ${\cal A}$, with very high probability, the solution goes into the set of good points at least once.
In some sense, this is consistent to the known result that averaging scheme can achieve the optimal rate.
In the second part, we bound the tail probability of the sub-optimality of the last iterate for nearly linear functions, from which we can bound the expectation of the sub-optimality.
Roughly speaking, we consider the events that $f(x_T)-f^*\geq \frac{GDk}{\sqrt{T}}$ and the last $T-i$ steps all lie out the set of good points, and bound its probaility by $\exp(-\Omega(k/\epsilon+\epsilon(T-i)))$. And by Union Bound we know the tail probability that $\Pr[f(x_T)-f(^*)\geq \frac{GDk}{\sqrt{T}}]\leq \exp(-\Omega(k/\epsilon))/\epsilon$, which is enough to get the optimal rate $O(\frac{GD}{\sqrt{T}})$. 

\begin{lemma}
\label{lm:large_subgradient}
For any $x\in X\setminus\Scal,\forall \nabla f(x)\in\partial f(x)$, one has
\begin{align*}
    |\nabla f(x)|>\frac{G}{\sqrt{T}}.
\end{align*}
\end{lemma}

\begin{proof}
We prove this statement by contradiction. 
Suppose there exists $x\in X\setminus \Scal$ such that $|\nabla f(x)|\leq \frac{G}{\sqrt{T}}$. 
By the convexity of $f$ and the definition of sub-gradient and let $x^*\in \K$ be a minimizer (arbitrarily if the minimizers are not unique), one has
\begin{align*}
    f(x^*)\geq f(x)+ \nabla f(x) (x-x^*),
\end{align*}
which implies that
\begin{align*}
    f(x)-f(x^*)\leq & \nabla f(x) (x^*-x)\leq \frac{GD}{\sqrt{T}}.
\end{align*}
This means $x\in \Scal$ and thus is a contradiction.
\end{proof}

Let $s_L=\inf_{x\in\Scal}x$ and $s_R=\sup_{x\in\Scal}x$.
\begin{definition}
We define the distance from a point to the set of good points as follows:
\begin{align*}
    \|x-\Scal\|=\left\{\begin{array}{ll}
n_{L}-x & \text { if } x<n_{L} \\
0 & \text { if } \quad n_{L} \leq x \leq n_{R} \\
x-n_{R} & \text { if } x>n_{R}
\end{array}\right.
\end{align*}
\end{definition}
Suppose we start from an arbitrary point $x_0\in X$ and the (random) sequence of the SGD algorithm with the fixed step size $\eta$ is denoted by $x_0,x_1,\cdots,x_T$, i.e. $x_{t+1}=\Pi_{X}(x_t-\eta\hat{g}_t)$.

\begin{lemma}
\label{lm:never_opt}
Given any $x_0\in X$.
Define $\tau_t:=\infty$ if SGD never goes back to the set of good points in the first $t$ steps, and $\tau_t:=\min_{i}\{0\leq i\leq t\mid x_{i}\in \Scal\}$ otherwise.
If $t\geq T$, we have that
\begin{align*}
    \Pr[\tau_t=\infty\mid x_0]\leq 2\exp(-\Omega(\sqrt{T})).
\end{align*}
\end{lemma}

\begin{proof}
Without loss of generality, we consider the case where $x_i> n_R$ for all $0\leq i\leq t$. We define a random variable $y_i=x_i-x_{i-1}$ to capture the movement of the solution for $0\leq i\leq t$.
Conditioning on $\tau =\infty$, i.e. $x_i>n_R$ for all $0\leq i\leq t$, we have that $\E[y_i]\leq -\eta \frac{G}{\sqrt{T}}=-4D/T$ for $i\geq 1$ by Lemma~\ref{lm:large_subgradient} (the projection only makes the expectation smaller).
By standard arguments, let $\F_i$ be the filtration and $\Tilde{y}_i=y_i-\E[y_i\mid \F_{i-1}]$.
It is easy to verify that $\{\Tilde{y}_{i}\}$ is a martingale difference sequence: 
\begin{gather}
    \E[\Tilde{y}_i\mid \F_{i}]=\E[y_i\mid \F_i]-\E[y_i\mid\F_i]=0.\\
    \E[|\Tilde{y}_i|]\leq G\eta <\infty.
\end{gather}

Obviously, one has $|\Tilde{y}_i|\leq G\eta= \frac{4D}{\sqrt{T}}$ by the third line of Assumptions $\A$.
As a result, $\E[\Tilde{y}^2_i\mid \F_{i-1}]=\E[y_i^2\mid \F_{i-1}]-(\E[y_i\mid\F_{i-1}])^2\leq \E[y_i^2\mid \F_{i-1}]\leq \eta G\E[y_i\mid \F_{i-1}]$.
Hence, we get the estimation $W_t=\sum_{i=1}^t \E[\Tilde{y_i}^2\mid \F_{i-1} ]\leq \eta G\left| \sum_{i=1}^t \E[y_i\mid \F_{i-1}]\right|$.
Let $\ell:=|\sum_{i=1}^{t}\E[y_i\mid \F_{i-1}]|$. By the Freedman's Inequality, one has:
\begin{align*}
    \Pr[\tau_t=\infty\mid x_0]\leq & \Pr[\sum_{i=1}^{t}y_i\geq -D]\\
    \leq & \Pr[|\sum_{i=1}^t \Tilde{y_i}|\geq \ell-D\wedge W_t\leq \eta G\ell]\\
    \leq & 2\exp\big(-\frac{(\ell-D)^2}{\eta Gl+\frac{4G\eta}{3}\ell}\big).
\end{align*}
If $t\geq T$, we know that $\ell= |\sum_{i=1}^{t}\E[y_i\mid \F_{i-1}]|\geq 4Dt/T\geq 4D$ and the statement follows directly by elementary calculation.
\end{proof}

\begin{theorem}
For any function satisfies the assumptions ${\cal A}$ and ${\cal B}$, one has
\begin{align*}
    \E[f(x_T)-f^*]=O(\frac{GD}{\sqrt{T}}),
\end{align*}
where $f^*=\min_{x\in X}f(x)$.
\end{theorem}
\begin{proof}
We try to bound the tail probability, that is $\Pr[f(x_T)-f^*\geq \frac{GDk}{\sqrt{T}}]$ for any $k\geq 3$.

We define $t:=\infty$ if SGD never goes in the set $\Scal$ and
let $t:=\max_i\{0\leq i\leq T\mid x_i\in \Scal\}$ otherwise.
One has
\begin{align*}
    \Pr[f(x_T)-f^*\geq \frac{GDk}{\sqrt{T}}]=& \sum_{i=0}^{T}\Pr[f(x_T)-f^*\geq \frac{GDk}{\sqrt{T}}\wedge t=i]+\Pr[f(x_T)-f^*\geq \frac{GDk}{\sqrt{T}}\wedge t=\infty]\\
    =& \sum_{i=0}^{T-1}\Pr[f(x_T)-f^*\geq \frac{GDk}{\sqrt{T}}\wedge t=i]+\Pr[f(x_T)-f^*\geq \frac{GDk}{\sqrt{T}}\wedge t=\infty],
\end{align*}
where the second equality follows from the fact that $\Pr[f(x_T)-f^*\geq \frac{GDk}{\sqrt{T}}\wedge t=T]=0$ by the definition of $\Scal$ and $k\geq 3$.

By Lemma~\ref{lm:never_opt}, we have
\begin{align*}
    \Pr[f(x_T)-f^*\geq \frac{GDk}{\sqrt{T}}\wedge t=\infty]\leq  \Pr[ t=\infty]\leq 2\exp(-\Omega(\sqrt{T})),
\end{align*}
which is negligible.

Now we begin to bound $\Pr[f(x_T)-f^*\geq \frac{GDk}{\sqrt{T}}\wedge t=i]$.
Similarly we use $y_i=x_i-x_{i-1}$ to capture the movement of the solution and without loss of generality. We assume that $x_j>n_R$ for all $i<j\leq T$, and by Assumption ${\cal B}$ we have $\E[y_i]\in [-c\eta \epsilon G,-\eta\epsilon G]$.

Let $\F_i$ be the filtration and $\Tilde{y}_i=y_i-\E[y_i\mid \F_i]$.
 We know that $W_{(i,T]}=\sum_{j=i+1}^T \E[\Tilde{y_i}^2\mid \F_{i-1} ]\leq \eta G| \sum_{j=i+1}^T \E[y_i\mid \F_{i-1}]|$.
Let $\ell=|\sum_{j=i+1}^{T}\E[y_j\mid \F_{j-1}]|$.
It is obvious that $\ell\geq c\eta\epsilon G(T-i)$ by the Assumption ${\cal B}$.

Conditioning on $f(x_T)-f^*\geq \frac{GDk}{\sqrt{T}}\wedge t=i$, it follows that $|\sum_{j=i+1}^T y_i|\geq \frac{D(k-1)}{\epsilon\sqrt{T}}$.
More specifically, as $x_i\in\Scal$ and thus $f(x_i)-f^*\leq \frac{GD}{\sqrt{T}}$, we have that $f(x_T)-f(x_i)\geq \frac{GD(k-1)}{\sqrt{T}}$ and further $|x_T-x_i|=|\sum_{j=i+1}^{T}y_j|\geq \frac{D(k-1)}{\epsilon \sqrt{T}}$ by the Assumption ${\cal B}$.

Hence we have
\begin{align*}
    \Pr[f(x_T)-f^*\geq \frac{GDk}{\sqrt{T}}\wedge t=i]
    \leq & \Pr[|\sum_{j=i+1}^T y_j|\geq \frac{D(k-1)}{\epsilon\sqrt{T}} ]\\
    \leq & \Pr[|\sum_{j=i+1}^T \Tilde{y}_j|\geq  \frac{D(k-1)}{\epsilon\sqrt{T}}+\ell \wedge W_{(i:T]}\leq \eta G \ell  ] \\
    \leq & 2\exp\left(- \frac{(\frac{D(k-1)}{\epsilon\sqrt{T}}+\ell)^2/2}{4\eta G \ell/3}\right)\\
    \leq & 2\exp(-\frac{3}{16\eta G}(\frac{D(k-1)}{\epsilon\sqrt{T}}+\ell))\\
    \leq & 2\exp\left(-\frac{3}{16\eta G}(\frac{D(k-1)}{\epsilon\sqrt{T}}+c\eta G\epsilon(T-i))\right)\\
    = & 2\exp(-\frac{3}{64}\frac{k-1}{\epsilon}-\frac{3}{16}c\epsilon (T-i)).
\end{align*}

And further 
\begin{align*}
    \Pr[f(x_T)-f^*\geq \frac{GDk}{\sqrt{T}}]=& \sum_{i=0}^{T-1}\Pr[f(x_T)-f^*\geq \frac{GDk}{\sqrt{T}}\wedge t=i]+\Pr[f(x_T)-f^*\geq \frac{GDk}{\sqrt{T}}\wedge t=\infty]\\
    \leq & \sum_{i=0}^{T-1}2\exp(-\frac{3}{64}\frac{k-1}{\epsilon}-\frac{3}{16}c\epsilon (T-i))+2\exp(-\Omega(\sqrt{T}))\\
    \leq & \frac{32}{3c\epsilon}\exp(-\frac{3}{64}\frac{k-1}{\epsilon})+2\exp(-\Omega(\sqrt{T})),
\end{align*}
where the last step follows from the fact that for any constant $C>0$ one has $\sum_{i=1}^T\exp(-C i)\leq \int_{i=0}^{T-1}\exp(-C i)\mathrm{d} i\leq 1/C$.

As a result, we have that
\begin{equation}
    \Pr[f(x_T)-f^*\geq x]=O(\exp(-x\lambda))
\end{equation}
where $\lambda=\Theta(\frac{\sqrt{T}}{GD})$. Our conclusion follows from
\begin{equation}
    \E[f(x_T)-f^* ]=\int_{0}^{GD}\Pr[f(x_T)-f^*\geq x]\mathrm{d} x=O(\lambda)=O(\frac{GD}{\sqrt{T}})
\end{equation}
\end{proof}

\subsection{Stationary distribution of a more general random walk}
The one-dimensional special example considered in \cite{koren2020open} is essentially a discrete one-dimensional random walk. Specifically, \cite{koren2020open} considered the final iterate of SGD on the function $f(x)=\epsilon |x|$ with fixed step size $\eta$ and a restricted SGD oracle only outputting $\pm 1$. Linearity together with fixed step sizes implies that the point $x_t$ can only appear at locations of the form $z\eta$ where $z$ is any integer. This property makes reducing the problem to a random walk possible. Since a one dimensional discrete random walk is essentially defined by its transition probabilities, it's natural to further consider random walks with varying transition probabilities.

In this subsection, we relax the assumption on the transition probabilities of the random walk, corresponding to running SGD on a Lipschitz convex function instead of a linear one, with the same restricted oracle in \cite{koren2020open}. In particular, we consider any 1-Lipschitz convex function $f(x): [0,1]\to R$, where $f(0)=0$ is the unique minimum. The SGD oracle can only return $1,-1$ at $x\in [0,1]$, and the expectation of the output must be equal to a sub-gradient of $f$ at $x$. 
We further assume that either the output distribution of the SGD oracle at any point $x$ is identical whenever being queried which allows the using of Markov chains. This more general example corresponds to a discrete one-dimensional random walk with monotone probabilities of 'moving left' at each point. To proceed, we need to introduce several definitions about Markov chains from \cite{freedman2017convergence}.

\begin{defn}[Finite Markov Chain]
A finite Markov chain with finite state space $\Omega$ and transition matrix $P$ is a sequence of random variables $X_t$ where
\begin{equation}
    \P\{X_{t+1}=y|X_t=x\}=P(x,y)
\end{equation}
and $P(x,y)$ are all non-negative with $\sum_y P(x,y)=1$. We further denote $P^t(x,y)=\P\{X_t=y|X_0=x\}$.
\end{defn}

\begin{defn}[Stationary Distribution]
A distribution $\pi$ is called a stationary distribution of a Markov chain $P$ if $ \pi P=\pi$.
\end{defn}

\begin{defn}[Irreducible Markov Chain]
A Markov chain is irreducible if for all states $x,y\in \Omega$, there exists a $t\ge 0$ such that $P^t(x,y)>0$.
\end{defn}

\begin{defn}[Aperiodic Markov Chain]
Let $\tau(x)=\{t>0|P^t(x,x)>0\}$ be the set of all time steps for which a Markov chain can start and end in a state $x$, then the period of $x$ is $gcd \tau(x)$. An irreducible Markov chain is called aperiodic if $gcd \tau(x)=1$ for any $x\in \Omega$.
\end{defn}

\begin{proposition}[Fundamental Theorem of Markov Chains]\label{prop1}
If a Markov chain $P$ is irreducible and aperiodic then it has a unique stationary distribution $\pi$.
\end{proposition}

Without loss of generality we assume $\sqrt{T}$ is an integer and denote $n=\sqrt{T}$, so the point can only move within the set $\{0,1/n,...,1\}$ when we run the restricted SGD on $f$. We have the following upper bound for the stationary distribution of the (induced) random walk.

\begin{theorem}
Under the above assumptions, for any 1-Lipschitz convex function $f(x): [0,1]\to R$ whose minimizer is $f(0)=0$. When we run SGD on $f$ whose oracle can only output $\pm 1$, with initial distribution supported on $\{0,1/n,...,1\}$ and step size $1/\sqrt{T}$, the asymptotic sub-optimality is $O(1/\sqrt{T})$.
\end{theorem}

\begin{proof}
We denote $p_i(t)$ to be the probability of the point at location $i/n$ at time $t$, and $a_i$ to be the probability of SGD outputing $1$ at location $i/n$. It's obvious this random process is a finite Markov chain, with the following transition matrix:

$$
A=\begin{bmatrix}
a_0&1-a_0&0&\dots & 0\\
a_1&0&1-a_1&\dots & 0\\
0&a_2&0&\dots & 0\\
\vdots& & & \ddots  & \vdots\\
&& & a_n& 1-a_n
\end{bmatrix}_{(n+1) \times (n+1)}
$$

It's easy to verify that the transition matrix is irreducible and aperiodic by observing that endpoints $\{0,1\}$ have positive probability to stay still, therefore having a unique stationary by proposition \ref{prop1}. Denote $p:=(p_0,...,p_n)$ to be the stationary, so that $p A=p$. By straightforward calculation we find that
\begin{align*}
    &p_1=\frac{1-a_0}{a_1}p_0\\
    &p_2=\frac{(1-a_0)(1-a_1)}{a_1 a_2}p_0\\
    &p_3=\frac{(1-a_0)(1-a_1)(1-a_2)}{a_1 a_2 a_3}p_0\\
    &\cdots
\end{align*}
This motivates the guess of solution $p_i=\frac{\prod_{j=0}^{i-1}(1-a_j)}{\prod_{j=1}^i a_j} p_0$, which is easily verified by induction on index $i$. The convexity nature of $f$ and the assumption that $f(0)=0$ is the unique minimum imply that $1/2\le a_0\le...\le a_n$. We would like to estimate the loss of the stationary. Denote $b_i:=2a_i-1\in [0,1]$, which is a sub-gradient at point $i/n$. By convexity we have that $f(i/n)\le \frac{1}{n}\sum_{j=0}^i b_j$, and we would like to show that
\begin{equation}
    \sum_{i=0}^n p_i \sum_{j=0}^i b_j=O(1)
\end{equation}
We discuss two possible cases. If $\sum_{i=0}^n b_i<1/2$, the above inequality is trivial as $\sum_{i=0}^n p_i=1$. If not, there exists a smallest $N\le n$ such that $\sum_{i=0}^N b_i\ge 1/2$. Because $\sum_{i=0}^N b_i\le 2$ and $\sum_{i=0}^n p_i=1$, we can control the sum of the first $N+1$ terms 
\begin{equation}
    \sum_{i=0}^N p_i \sum_{j=0}^i b_j\le 2\sum_{i=0}^N p_i\le 2
\end{equation}
For the rest part, we upper bound $p_i$ as 
\begin{equation}
    p_i=p_0 \prod_{j=0}^{i-1} \frac{1-b_j}{1+b_{j+1}}   \le p_0\prod_{j=0}^{i-1} \frac{1-b_j}{1+b_j} 
\end{equation}
and further by $\log \frac{1-b_j}{1+b_j} \le -b_j$
\begin{equation}
    p_i \sum_{j=0}^i b_j\le  p_0 e^{-\sum_{j=0}^{i-1} b_j} \sum_{j=0}^i b_j \le  p_0 e^{1-\sum_{j=0}^i b_j} \sum_{j=0}^i b_j
\end{equation}
Notice that $b_i\ge \frac{1}{2(N+1)}$ for any $i>N$ by its monotony, we have that
\begin{equation}
    \sum_{i=0}^N e^{-\sum_{j=0}^i b_j} \sum_{j=0}^i b_j\le 2(N+1) \int_0^{\infty} e^{-x} x dx + 4(N+1)\max_x e^{-x}x=2(N+1)(1+\frac{2}{e})
\end{equation}
and further
\begin{equation}
    \sum_{i=N+1}^n p_i \sum_{j=0}^i b_j\le 2e p_0 (N+1)(1+\frac{2}{e})
\end{equation}
Recall that $\sum_{i=0}^{N-1}b_i<1/2$, for any $i<N$, we have the following control by concavity of $g(x)=\frac{1-x}{1+x}$
\begin{equation}
    p_i\ge p_0\prod_{j=0}^i \frac{1-b_j}{1+b_j} \ge \frac{p_0}{3}
\end{equation}
This implies that $\frac{N p_0}{3}\le 1$ and $\sum_{i=N+1}^n p_i \sum_{j=0}^i b_j\le 24e$. Combing results for both cases, we conclude that the loss of stationary is $O(1/\sqrt{T})$:
\begin{equation}
    \sum_{i=0}^n p_i f(i/n)\le \frac{1}{n} \sum_{i=0}^n p_i \sum_{j=0}^i b_j\le \frac{2+24e}{\sqrt{T}}
\end{equation}
So far, we have shown that the stationary has optimal rate $O(1/\sqrt{T})$.
\end{proof}
\section{Conclusion}
In this paper, we analyze the convergence rate of the final iterate of SGD running on non-smooth strongly convex/ lipschitz convex functions, when the dimension $d$ is seen as a relevant parameter. We prove $\Omega(\log d/\sqrt{T})$ and $\Omega(\log d/T)$ lower bounds for the sub-optimality of SGD minimizing non-smooth general convex and strongly convex functions respectively with standard step size schedules. We also prove a tight $O(1/\sqrt{T})$ upper bound for one-dimensional (nearly) linear functions, a more general setting than \cite{koren2020open}. This work is the first, to the best of our knowledge, that characterizes the dependence on dimension in the general $d\le T$ setting, advancing our knowledge on the final iterate convergence of SGD. Our results reveal a surprising advantage of the running average schedule, that it enjoys dimension-free convergence rate while the final iterate still (slightly) suffers from the curse of dimensionality. Our general lower bounds together with the upper bounds for one-dimensional special cases suggest that the right rate is $\Theta(\log d/\sqrt{T})$. We leave dimension-dependent upper bounds for future works.

\bibliography{Xbib}
\bibliographystyle{plainnat}

\end{document}